\documentclass[letterpaper, 10 pt, conference]{ieeeconf}  % Comment this line out if you need a4paper
\IEEEoverridecommandlockouts                              % This command is only needed if 
\let\labelindent\relax
\usepackage{enumitem}
\usepackage{graphicx}
\usepackage{amssymb,mathtools,amsfonts}
\usepackage{gensymb}
  
\usepackage{amsthm}  
\usepackage{caption, subcaption}
\usepackage{hyperref}
\usepackage{cleveref}
\usepackage{multirow}
\usepackage{color}
\usepackage{psfrag}
\usepackage{empheq}
\usepackage{bm} % bold italic symbols in math mode
\usepackage{url}
\usepackage{float}
\usepackage{comment}
\usepackage{alphalph}
\usepackage[short]{optidef}
\usepackage[dvipsnames]{xcolor}
\usepackage{tikz}
\usepackage{suffix}

\usetikzlibrary{positioning}
\usetikzlibrary{arrows}
\usetikzlibrary{babel}
\usetikzlibrary{automata}
\usetikzlibrary{shapes, snakes}
\usetikzlibrary{positioning, fit, calc}
\usetikzlibrary{decorations.pathreplacing,calligraphy}
\usetikzlibrary{patterns}
\usetikzlibrary{patterns.meta}
\usepackage{pgfplots}
\pgfplotsset{compat=newest}
\usepackage{pgfplotstable}
\usepackage{pgf}

\newcommand{\R}{{\mathbb{R}}}
\newcommand{\dd}{{\mathrm{d}}}

\newcommand{\activeset}{{\mathcal{A}}}
\newcommand{\A}{{\mathcal{S}}} % active set
\newcommand{\W}{{\mathcal{W}}} % active set
\newcommand{\I}{{\mathcal{I}}} % active set

\newcommand{\ns}{{n_\mathrm{s}}}

\newcommand{\Nfe}{{N_\mathrm{fe}}}

\newcommand{\paren}[1]{\left( #1 \right)}
\newcommand{\brac}[1]{\left[ #1 \right]}
\newcommand{\cbrac}[1]{\left\{ #1 \right\}}
\WithSuffix\newcommand\paren*[1]{( #1 )}
\WithSuffix\newcommand\brac*[1]{[ #1 ]}
\WithSuffix\newcommand\cbrac*[1]{\{ #1 \}}

\newcommand{\vo}{\vec{o}\@ifnextchar{^}{\,}{}}

\newcommand{\Set}[2]{\left\{\, #1 \mid #2\,\right\}}
\WithSuffix\newcommand\Set*[2]{\{\, #1 \mid #2\,\}}
\newcommand{\norm}[1]{{\left\Vert#1\right\Vert}}
\newcommand{\transp}[1]{{#1}^\top}

\newcommand{\C}{{\mathcal{C}}}
\newcommand{\xdot}{{\dot{x}}}
\newcommand{\proj}[2]{\mathrm{P}_{#1}\paren{#2}}
\WithSuffix\newcommand\proj*[2]{\mathrm{P}_{#1}\paren*{#2}}
\newcommand{\tancone}[2]{\mathcal{T}_{#1}\paren{#2}}
\newcommand{\normcone}[2]{\mathcal{N}_{#1}\paren{#2}}

\newcommand{\nupds}{\nu}
\newcommand{\interior}[1]{\mathrm{int}(#1)}
\newcommand{\hatt}{\hat{t}}

\DeclareMathOperator*{\argmin}{arg\,min}
\newcommand{\AND}{\mathbin{\wedge}}
\newcommand{\OR}{\mathbin{\vee}}

\newtheorem{theorem}{Theorem}
\newtheorem{example}{Example}

\newtheorem{proposition}[theorem]{Proposition}

\begin{document}

\title{Finite Elements with Switch Detection for Numerical Optimal Control of Projected Dynamical Systems}
\author{
  Anton Pozharskiy$^{1}$, Armin Nurkanovi\'c$^{1}$, Moritz Diehl$^{1,2}$
  \thanks{This research was supported by DFG via Research Unit FOR 2401, project 424107692 and 525018088, by BMWK via 03EI4057A and 03EN3054B, and by the EU via ELO-X 953348.}
  \thanks{
  $^1$Department of Microsystems Engineering (IMTEK), 	
  $^2$Department of Mathematics, University of Freiburg, Germany,  
  \sloppy\texttt{\{anton.pozharskiy,armin.nurkanovic,moritz.diehl\} @imtek.uni-freiburg.de}
	}
}
\maketitle
\thispagestyle{empty} % Removes the page number in the first page
% \copyrightnotice
\begin{abstract}
  The Finite Elements with Switch Detection (FESD) method~\cite{Nurkanovic2022} is a highly accurate direct transcription method for optimal control of several classes of nonsmooth dynamical systems. 
  This paper extends the FESD method to Projected Dynamical Systems (PDS) and first-order sweeping processes with time-independent sets.
  This method discretizes an equivalent dynamic complementarity system and exploits the particular structure of the discontinuities present in these systems. 
  In the FESD method, allowing integration step sizes to be degrees of freedom, and introducing additional complementarity constraints, enables the exact detection of nonsmooth events.
  In contrast to the standard fixed-step Runge-Kutta methods, this approach allows for the recovery of full-order integration accuracy and the correct computation of numerical sensitivities.
  Numerical examples illustrate the effectiveness of the proposed method in an optimal control context.
  This method and the examples are included in the open-source software package~\texttt{nosnoc}.
\end{abstract}
\section{Introduction}
In this paper, we develop direct transcription methods for optimal control problems subject to Projected Dynamical Systems (PDS), and a particular class of First-Order Sweeping Processes (FOSwP).
Both of these systems belong to the class of discontinuous dynamical systems, where the trajectories are constrained to a set $\mathcal{C}$. 
In the case of PDS, first introduced by Dupuis and Nagurney~\cite{Dupuis1993}, the system evolves according to some smooth Ordinary Differential Equation (ODE) on the interior of the set $\mathcal{C}$.
If the trajectory reaches the boundary of $\C$, the vector field is projected onto the tangent cone to $\C$ at this point, such that the trajectory stays in $\C$.

A closely related concept is the first-order sweeping process, introduced by Moreau~\cite{Moreau1971}, in which we have a nominal trajectory and a possibly moving and state-dependent set $\mathcal{C}$.
If the trajectory is in contact with the boundary of the set, the vector field is modified such that the trajectory is swept together with the set.
There are clear structural similarities between these two notions, as both are dynamic systems constrained to a set.
These similarities are more than superficial if the set is not time-dependent and under mild assumption on the dynamics~\cite{Brogliato2006}.
%PDS are the combination of a smooth ordinary differential equation (ODE) and a feasible set in which the system must stay.
%This is accomplished by projecting the vector field of the ODE onto the tangent cone of the feasible set in a pointwise fashion.
%FOSwP have been extensively studied as a modeling technique for non-smooth phenomena.
%The intuition behind FOSwP is that a moving set ``sweeps'' the state along by keeping it within its boundaries.
%Many extensions exist to this basic FOSwP including the addition of additional ``perturbations'' to the dynamics, we point particularly to the bibliography of~\cite{Brogliato2020} for further reading on the topic.

The applications of PDS and FOSwP are various.
PDS have been applied to economic equilibrium problems~\cite{Nagurney1995}, vaccination strategies~\cite{Cojocaru2007}, and more recently to several kinds of controllers including hybrid-integrator gain systems~\cite{Sharif2019} and anti-windup controllers~\cite{Hauswirth2020}.
FOSwP models have also been an active area of research in recent years including the development of a Pontryagin-style maximum principle being derived for Mayer problems~\cite{Arroud2018}.
Applications include optimal control of crowd dynamics~\cite{Cao2021a}, unmanned marine surface vehicles~\cite{Cao2021,Mordukhovich2023}, and soft robotics~\cite{Colombo2021}.
Several of these applications model interesting physical systems via FOSwP with time-invariant feasible sets~\cite{Mordukhovich2023}.

For smooth dynamical systems, time-stepping methods are usually used for the direct transcription of optimal control systems.
Several existing algorithms could fall under the umbrella of time-stepping methods for PDS and FOSwP.
In particular, for PDS there are several time-stepping algorithms proposed in Chapter 4 of~\cite{Nagurney1995}, which apply a projection to each instance of the derivative evaluation in a fixed time step Runge-Kutta (RK) method.
For FOSwP, the initial papers of Moreau~\cite{Moreau1971} presented the so-called ``catching-up algorithm''.

However, the difficulty of both the simulation, and use in direct optimal control, of these methods comes from the fact that these systems have a nonsmooth and possibly discontinuous vector field.
The nonsmoothness of the derivative leads to reduced integration order accuracy in fixed step size RK methods~\cite{Acary2008}.
Furthermore, the nonsmoothness may lead to incorrect sensitivities (i.e., derivatives of the integration state transition map with respect to parameters) and convergence to spurious solutions as discussed in~\cite{Nurkanovic2020,Stewart2010}.

In PDS, the event of the trajectory entering or leaving the boundary of the feasible set $\mathcal{C}$ is called a switch, and the corresponding time is a switching time. 
The trajectories of the system between switching times are smooth functions of time. 
If a numerical method can correctly identify the switching times, then the limitations of time-stepping methods can be overcome.
This is achieved for instance with the recently introduced Finite Elements with Switch Detection (FESD) method, which recovers the accuracy of an underlying RK method and computes correct numerical sensitivities. 
This method has so far been developed for piecewise smooth systems~\cite{Nurkanovic2022} and nonsmooth mechanical systems with impacts and friction~\cite{Nurkanovic2024}.
However, it cannot be directly applied to PDS and FOSwP. 
% \anton{this feels repetitive} This paper extends the FESD discretization method to these two classes of systems and describes its application to optimal control.

\paragraph*{Contributions}
In this paper, we extend the FESD method for both PDS and a class of FOSwP.
We exploit the fact that these two systems are equivalent to a specific Dynamic Complementarity System (DCS).
The FESD method is then developed for this DCS. 
We study the continuity properties of the algebraic and differential states of the DCS.
These are exploited to derive appropriate cross-complementarity conditions in the FESD discretization, which ensure exact switch detection.
It is proven that the switches are correctly identified in the discretized system.
Moreover, if there are no switches, we need to introduce step equilibration conditions, which remove the degrees of freedom in the step sizes.
In this paper, we propose a new step equilibration formulation, which is significantly less nonlinear than the initial formulation in~\cite{Nurkanovic2024,Nurkanovic2022}.
We verify the claims of recovered order accuracy of the discretization method via simulation.
Finally, we apply our discretization to a collaborative planar manipulation task, which is formulated as an optimal control problem.
\paragraph*{Notation}
Regard a closed set $\C\subseteq \R^n$. 
The boundary of the set is denoted by $\partial \C$, and its interior by $\interior{\C}$.
The tangent cone to $\C$ at $x$, denoted by $\tancone{\C}{x}$, is the set of all vectors $d\in \R^n$ for which there exists sequences $\{x_i\} \in \C$ and $\{t_i\}, t_i >0$, with $x_i \to x$ and $t_i\to0$, such that ${d = \lim_{i\to \infty} \frac{x_i-x}{t_i}}$.
The normal cone to $\C$ at $x$ is defined by $\normcone{\C}{x} = \Set{v\in\R^n}{\langle v,d\rangle \le 0, \forall d\in\tancone{C}{x}}$, for  $x\in \C$ and is defined to be the empty set for $ x\notin \C$.
Let $K\subseteq \R^n$ be a closed convex set.
The projection operator is defined via a convex optimization problem ${\proj{K}{x} =\argmin_{s \in K}\frac{1}{2}\norm{s-x}^2_2}$.
The set $\C$ is said to be prox-regular if for any $x\in\C$ the distance to the set is continously differentiable on $O\setminus\C$, where $O$ is an open neighborhood of $x$.

\section{Constrained dynamical systems}
In this section, we introduce details of the mathematical formalisms for the three equivalent kinds of systems studied in this paper.
We then introduce the optimal control problem that is the focus of this paper.
Finally, we provide a proposition that describes the continuity properties of the regarded classes of systems.
\subsection{Projected Dynamical Systems (PDS)}
\label{sec:PDS}
We regard the PDS:
\begin{equation}
  \label{eq:PDS}
  \xdot(t) = \proj{\tancone{\C}{x(t)}}{f(x(t))},
\end{equation}
with $x(0) \in \C$ and $f:\R^{n_x}\rightarrow\R^{n_x}$ is at least twice continuously differentiable.
Under these assumptions the state $x(t)$ stays in $\C$ for all time $t \in [0, \infty)$.
We further assume that the feasible set is finitely defined by ${\C = \Set*{x\in\R^n}{c(x)\ge 0}}$ with $c:\R^{n_x}\rightarrow \R^{n_c}$ being also at least twice continuously differentiable.
We assume that the Linear Independence Constraint Qualification (LICQ) is satisfied at all $x \in \C$, i.e., the columns of $\nabla c_k(x)$ for ${k\in\activeset(x) = \Set*{i\in \cbrac*{1,\ldots,n_c}}{c_i(x)=0}}$ are linearly independent.
Under LICQ, the tangent cone of $\C$ is equal to the convex polyhedral cone ${\tancone{\C}{x} = \Set*{v\in\R^{n}}{\transp{\nabla c_i(x)}v \ge 0, i \in\activeset(x)}}$~\cite{Rockafellar1997}.
\subsection{First-Order Sweeping Processes}
In this paper, we treat a class of differential inclusions closely related to PDS, which are equivalent to a perturbed first-order sweeping process with a time-independent $\C$: 
\begin{align}\label{eq:MSP}
  \xdot(t)&\in f(x(t)) -\normcone{\C}{x(t)},
\end{align}
with $x(0) \in \C$, $f:\R^{n_x}\rightarrow\R^{n_x}$ twice continuously differentiable, and $\C$ finitely defined as before.
\subsection{PDS and FOSwP as Dynamic Complementarity Systems}
Regard the dynamic complementarity system:
\begin{subequations}
  \label{eq:GCS}
  \begin{align}
    \xdot(t) &= f(x(t)) + \nabla c(x(t))\lambda(t),\\
    0 &\le c_i(x(t)) \perp \lambda_i(t) \ge 0, \quad i=1,\ldots,n_c, \label{eq:GCS:comps}
  \end{align}
\end{subequations}
with the differential state $x\in\R^{n_x}$ and the algebraic state ${\lambda\in\R^{n_c}}$. The functions  $f:\R^{n_x}\rightarrow\R^{n_x}$ and $c:\R^{n_x}\rightarrow\R^{n_c}$ are assumed to be at least twice continuously differentiable. 
The notation $c_i(x(t)) \perp \lambda_i(t)$ represents the algebraic constraint $c_i(x(t))\lambda_i(t) = 0$, i.e., at least one of $c_i(x(t))$ or $\lambda_i(t)$ must be zero.
%A DCS in this form is sometimes called a gradient complementarity system~\cite{Brogliato2006}.
\begin{figure}[t]
  \centering
  \vspace{0.1cm}
  \begin{tikzpicture}
  \begin{axis}[
    width=\linewidth, %axis equal image,
    height=(1.6/2.1)*\linewidth,
    xmin=-1.2,xmax=2.1,
    ymin=-1.3,ymax=1.6,
    axis x line=center, axis y line=center,
    xlabel={$x_1$},
    xlabel style={below right},
    ylabel={$x_2$},
    ylabel style={above left},
    axis line style={shorten >=-3pt},
    view={0}{90}
    ]
    \draw[-,color=Mahogany, ultra thick] (-1, 0) arc ({180}:{270}:1) -- ({sqrt(3)}, -1) arc ({-30}:{45}:2);
    \addplot[only marks, mark=*, ultra thick, color=Mahogany] ({sqrt(2)},{sqrt(2)}) node[below left] {$x(0)$};
    \node[color=Mahogany,] at (1.65, -0.5) {$x(t)$};
    \draw[dashed,color = gray, thick] (-1.2,-1) -- (2.1, -1) node[pos=0.075, above, yshift=1pt] {$c(x)\mathord{=}0$};
    \addplot3[domain=-1:2, y domain=-0.6:1.5,color=OliveGreen,-stealth, samples = 10, samples y=6,quiver={u=y, v=-x, scale arrows=0.15},thick]{0};
    \addplot[domain=0:2,color=OliveGreen,-stealth,samples=7,quiver={u=y, v=0, scale arrows=0.15}, thick]{-1};
    \addplot[domain=0.33:2,color=Plum,-stealth,samples=6,quiver={u=y, v=-x, scale arrows=0.15}, thick]{-1};
    \addplot[domain=-1:-.33,color=OliveGreen,-stealth,samples=3,quiver={u=y, v=-x, scale arrows=0.15},thick]{-1};
    \fill[pattern={Lines[angle=-45,distance=5.0]}, pattern color=gray] (-1.2,-1) rectangle (2.1,-2);
  \end{axis}
\end{tikzpicture}
%%% Local Variables:
%%% mode: latex
%%% TeX-master: "../cdc2024_fesd_pds"
%%% End:
  \vspace{-0.9cm}
  \caption{Plot of a trajectory of the PDS described in \Cref{ex:simple-pds} along with its projected gradient field in green, and original gradient field in purple.}
  \vspace{-0.7cm}
  \label{fig:ex1-sol}
\end{figure}
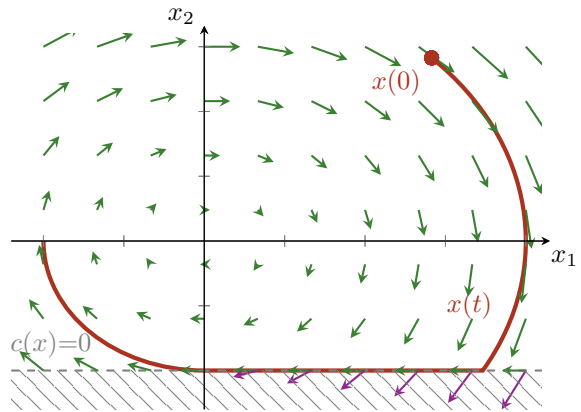

To develop the FESD discretization in future sections we rely on the equivalence of both \cref{eq:PDS} and \cref{eq:MSP} with \cref{eq:GCS}.
We formulate this as the following proposition.
\begin{proposition}
  Regard ${\C = \Set{x\in\R^{n_x}}{c(x) \ge 0}}$, a prox-regular set, and suppose that LICQ holds for all $x\in\C$.
  The following systems are equivalent:
  \vspace{-0.15cm} 
  \begin{enumerate}[label=\roman*)]
  \item the PDS, $\xdot(t) = \proj{\tancone{C}{x(t)}}{f(x(t))}$,
  \item the differential inclusion, $\xdot(t)\in f(x(t)) -\normcone{\C}{x(t)}$, 
    under the condition that for ${t\in [0,T]}$, $\xdot(t)$ is of minimum norm, i.e.,
    $\xdot(t) = \argmin_{v\in f(x(t)) -\normcone{\C}{x(t)}} \norm{v}$,
  \item DCS, \cref{eq:GCS}.
  \end{enumerate}
%  \vspace{-0.2cm} 
\end{proposition}
\vspace{-0.25cm}
A proof of this for convex sets can be found in~\cite[Theorem 1]{Brogliato2006}.
In this theorem there are technical assumptions made on the behavior of $f(x)$, which in the setting here is implied by the continuity and single-valuedness of $f(x)$~\cite[Example 12.28]{Rockafellar1997}.
This equivalence has been extended to prox-regular sets in \cite[Proposition 5]{Serea2003}.
\begin{example}\label{ex:simple-pds}
  To illustrate the equivalence regard the following PDS: $x\in \R^2$, $f(x) = (x_2,-x_1)$, and $c(x) = x_2+1$.
  The tangent cone of $\C$ reads as: ${\tancone{\C}{x} = \{v\in\R^2 \mid (0,1)^\top v\geq 0\}}$, if $x_2+1 = 0$, and $\tancone{\C}{x} = \R^2$, if $x_2+1 > 0$.
  The normal cone for the set $\C$ is $\normcone{\C}{x} = \Set{v \in \R^2}{ (0,\lambda), \lambda\geq0}$, if $x_2+1 = 0$ and $\normcone{\C}{x} =\cbrac{0}$, if $x_2+1 > 0$.
  Finally, we write the DCS equivalent to this system:
  \begin{equation*}
    \xdot = (x_2, -x_1 + \lambda),\   0 \le x_2+1 \perp \lambda \ge 0.
  \end{equation*}
  \Cref{fig:ex1-sol} shows an example trajectory of this system and the corresponding vector fields.
\end{example}
\subsection{Optimal Control of Constrained Dynamical Systems}
The FESD discretization we propose is then used to discretize and approximately solve Optimal Control Problems (OCPs) with dynamics governed by \cref{eq:GCS}:
\begin{mini!}[3]
  {\substack{x(\cdot),u(\cdot),\lambda(\cdot)}}
  {\int_{0}^{T} L(x(t),u(t))\dd t +  R(x(T))}{\label{eq:ocp}}
  {}
  \addConstraint{x(0)}{=x_0}
  \addConstraint{\xdot(t)}{= f(x(t),u(t)) + \nabla c(x(t))\lambda(t),\ }{t \in [0, T]}
  \addConstraint{0}{\leq c(x(t))\perp \lambda(t)\geq 0,\ }{t \in [0,T]}
  \addConstraint{0}{\geq g_{\mathrm{p}}(x(t),u(t)),\ }{t \in [0,T]}
  \addConstraint{0}{\geq g_{\mathrm{t}}(x(T)).}
\end{mini!}
With the Lagrange cost term $L:\R^{n_x}\times\R^{n_u}\rightarrow\R$, Mayer cost term $R:\R^{n_x}\rightarrow\R^{n_u}$, initial state $x_0\in\R^{n_x}$, path constraints $g_{\mathrm{p}}:\R^{n_x}\times\R^{n_u}\rightarrow\R^{n_{g_\mathrm{p}}}$, terminal constraints ${g_{\mathrm{t}}:\R^{n_x}\rightarrow\R^{n_{g_\mathrm{t}}}}$ and controls $u(t)\in\R^{n_u}$.
We introduce here exogenous controls, which were not in the original dynamics but, under mild conditions, the systems remain equivalent~\cite{Serea2003}.
\subsection{Types of Switches in PDS}\label{sec:types-of-switches}
It is useful to discuss the possible discontinuities present in this class of systems, in order to motivate our extension of FESD.
We note that based on the smoothness assumption on $f(x)$, for $x\in\interior{\C}$, $\xdot$ is smooth.
Therefore we only need to regard switches that occur on interactions of the solution $x(t)$ with $\partial\C$.
We define the following sets which define the the mode of the system at a given $x,\ \lambda$:
\begin{subequations}
  \label{eq:activity-sets}
  \begin{align*}
    \A(x,\lambda) = \Set{k\in\cbrac{1,\ldots, n_c}}{c_k(x) = 0,\ \lambda_k > 0},\\
    \W(x,\lambda) = \Set{k\in\cbrac{1,\ldots, n_c}}{c_k(x) = 0,\ \lambda_k = 0},\\
    \I(x,\lambda) = \Set{k\in\cbrac{1,\ldots, n_c}}{c_k(x) > 0,\ \lambda_k = 0}.
  \end{align*}
\end{subequations}
These are the index sets for constraints that are strongly active, weakly active, and inactive, respectively.
For a given $x,\ \lambda$, these sets form a partition of $\cbrac{1,\ldots,n_c}$.
We say that a constraint $k$ becomes active at $\hatt$ if for some $\epsilon >0$, $k\in\I(x(t),\lambda(t))$ for $t\in [\hatt-\epsilon,\hatt)$ and ${k\in\A(x(t),\lambda(t))\cup\W(x(t),\lambda(t))}$ for $t\in [\hatt,\hatt+\epsilon)$.
Similarly a constraint $k$ becomes inactive at $\hatt$ if for some $\epsilon >0$, $k\notin\I(x(t),\lambda(t))$ for $t\in [\hatt-\epsilon,\hatt)$ and $k\in\I(x(t),\lambda(t))$ for $t\in [\hatt,\hatt+\epsilon)$.
We now prove the following proposition on the behaviors of $c(x(t))$ and $\lambda(t)$. 
\begin{proposition} \label{prop:switches}
  Regard a DCS of the form \cref{eq:GCS} and its solutions $x(t)$, $\lambda(t)$ on the interval $t\in [0, T]$.
  The following statements are true:
  \begin{enumerate}
  \item $c(x(t))$ is a continuous function of time.
  \item $\lambda_k(t)$ is discontinuous at $\hatt\in[0,T]$ if the $k^{\mathrm{th}}$ constraint becomes active at $\hatt$, and $f(x(\hatt)) \notin \tancone{\C}{x(\hatt)}$.
  \item $\lambda_k(t)$ is continuous at $\hatt\in[0,T]$ if the $k^{\mathrm{th}}$ constraint becomes inactive at $\hatt$.
  \end{enumerate}
\end{proposition}
\begin{proof}
  1) follows from the absolute continuity of $x(t)$~\cite{Nagurney1995} and smoothness of $c(x)$.
  For 2), we show that the left limit and the right limit of $\lambda(t)$ are not equal at the point when the $k^{\mathrm{th}}$ constraint becomes active.
  We treat a single switch at index $k$. 
  By differentiating \cref{eq:GCS:comps} we obtain:
  \begin{equation*}
    0\le \dot{c}=\transp{\nabla c_k}(f(x))+ {\nabla c_k}\lambda_k)\perp \lambda_k\ge 0.
  \end{equation*}
  This yields the closed-form solution:
  \begin{equation}
    \label{eq:lambda-closed-form}
    \lambda_k = \max\paren*{0,\paren*{\transp\nabla c_k{\nabla c_k}}^{-1}\nabla \transp c_kf(x,u)}.
  \end{equation}
  The left limit is ${\lim_{t\rightarrow \hatt^-} \lambda_k(t) = 0}$ by \cref{eq:GCS:comps}.
  The right limit is ${\lim_{t\rightarrow \hatt^+} \paren*{\transp{\nabla c_k}\nabla c_k}^{-1}\nabla \transp c_k f(x(t))}$ via \cref{eq:lambda-closed-form}, which is zero if $\lim_{t\rightarrow \hatt^+} \transp \nabla c_k f(x(t))=0,$ which implies that $f(x(t))\in \tancone{\C}{x(t)}$ at $\hatt$.
  Therefore, if this is not the case, the right and left limits of $\lambda_k(t)$ are different at $\hatt$.
  Finally, to prove 3) we first observe that $\dot{c}_k$ must become strictly positive for some interval $[\hatt,\hatt+\epsilon]$ with $\epsilon>0$ in order to for $c_k(x)=0$ to become inactive.
  Therefore the leaving of the boundary depends on the sign of $\nabla \transp{c_k} f(x,u)$.
  This combined with \cref{eq:lambda-closed-form} implies that at the time $\hatt$ where $c_k(x)=0$ becomes inactive, $\lim_{t\rightarrow \hatt^-}\transp{\nabla c_k(x)}f(x)=0$ and therefore the left limit $\lim_{t\rightarrow \hatt^-}\lambda(t)=0$.
  The right limit of $\lambda_k(t)$ at $\hatt$ is 0 by \cref{eq:GCS:comps}.
\end{proof}
A useful consequence of this proposition is that in all of these switching cases, the left limits of $c_k(x(t))$ and $\lambda_k(t)$  must be zero if there is a switch in the $k^\mathrm{th}$ constraint at~$\hatt$.
\section{The FESD Discretization for PDS and FOSwP}
In this section, we introduce an extension to the FESD approach to accurately discretize the DCS described in \cref{eq:GCS}.
We first write the RK discretization used and then introduce the so-called cross-complementarity and step-equilibration constraints to complete the method.
We also include a formalization of the notion of correct step identification inspired by the notions of the types of switches described in \Cref{prop:switches}.

\subsection{Runge-Kutta discretization with fixed time steps}
We regard a time interval $[0, T]$ divided into $\Nfe \ge 2$ finite elements which are each a single-step RK discretization.
In a discretized OCP, this could be a single control interval with a constant $u$.
We begin the definition of the FESD discretization by defining a fixed step size RK discretization scheme with $\ns$ stages that are defined by a Butcher tableau~\cite{Hairer1996} with entries $a_{i,j}$, $b_j$, $c_i$ for $i,j = 1,\ldots,\ns$.
For ease of exposition, we assume that the final grid point of the RK scheme is at the end of the step, i.e., $c_\ns = 1$.
However, this can be extended in a manner similar to~\cite{Nurkanovic2022}.
This time interval (e.g. the control stage) is partitioned by the grid points $(t_1,t_2,\ldots,t_\Nfe)$ with the $t_1=0$, $t_{n+1} = t_n+h_n$.
This yields a discrete-time system in the form:
\begin{subequations}
  \label{eq:discrete-dcs}
  \begin{align}
    &x_{n,i}\! = \! x_{n,0} + h_n\sum_{j=1}^{\ns}a_{i,j}(f(x_{n,j},u)+ {\nabla c(x_{n,j})}\lambda_{n,j}),\\
    &0 \le c(x_{n,i}) \perp \lambda_{n,i} \ge 0. \label{eq:discrete-dcs:comps}
  \end{align}
\end{subequations}
In \cref{eq:discrete-dcs}, $n = 1,\ldots,\Nfe$, $i = 1,\ldots,\ns$, $x_{n,i} \in\R^{n_x}$, ${\lambda_{i,j}\in \R^{n_c}}$, and $f(\cdot),\ c(\cdot)$ as in \cref{eq:GCS}.

To go from this discretization to the FESD discretization, we allow the length of each finite element $h_n$ to vary and add two sets of additional constraints which are called cross-complementarity and step-equilibration.
These additional constraints are included to enforce the correct detection of switches and to remove the spurious degrees of freedom $h$ when there are no switches.

\subsection{Cross Complementarity}
For the $n^{\mathrm{th}}$ RK integration step to recover its full order accuracy we require that the right-hand side of the ODE on that step is sufficiently smooth~\cite{Schumacher2004}.
This is equivalent to the fact that we have no switches on the interior of the finite element which can be written as:
\begin{subequations}
  \label{eq:set-disjointness}
  \begin{align*}
    \A(x_{n,i},\lambda_{n,i})\cap \I(x_{n,j},\lambda_{n,j})=\emptyset,\\
    \A(x_{n-1,\ns},\lambda_{n-1,\ns})\cap \I(x_{n,j},\lambda_{n,j})=\emptyset,\\
    \A(x_{n,i},\lambda_{n,i})\cap \I(x_{n-1,\ns},\lambda_{n-1,\ns})=\emptyset,
  \end{align*}
\end{subequations}
for $n= 2,\ldots,\Nfe$, $i,j=1,\ldots,\ns$.
These intersections encode that no constraint can be both active and inactive on the same finite element at different stage points.
To accomplish this we introduce the set of complementarity constraints:
\begin{subequations}
  \label{eq:cross-complementarity}
  \begin{align}
    0 \le c(x_{n,i}) \perp \lambda_{n,j} \ge 0,\label{eq:cross-complementarity:a}\\
    0 \le c(x_{n-1,\ns}) \perp \lambda_{n,j} \ge 0,\label{eq:cross-complementarity:b}\\
    0 \le c(x_{n,i}) \perp \lambda_{n-1,\ns} \ge 0,\label{eq:cross-complementarity:c}
  \end{align}
\end{subequations}
with $n = 2,\ldots,\Nfe$, $i = 1,\ldots,\ns$, and $j = 1,\ldots,\ns$, which encompass both the original complementarity pairs in \cref{eq:discrete-dcs}, and the new constraints which enforce implicit identification of the switching times.
For brevity we write $\A_{n,i}=\A(x_{n,i},\lambda_{n,i})$, and the same for $\W$ and $\I$.
Recall that the conditions for a switch described in \Cref{prop:switches} require that the left limit of both $c_k(x(t))$ and $\lambda_k(t)$ is zero. 
In the discrete-time setting, if a switch occurs, we require that these functions are zero at the right boundary points of a finite element, i.e., $c_k(x_{n,n_s}) = 0$ and $\lambda_{n,n_s,k} = 0$, where $\lambda_{n,n_s,k}$ is the $k^{\mathrm{th}}$ element of $\lambda_{n,n_s}$.
This can equivalently be written as $k \in \W_{n,n_s}$.
This is formalized in the following proposition.
\begin{proposition}
  \label{prop:switch-identification}
  For any feasible solution to the discrete FESD problem, the following conditions hold for any $n\in\cbrac{2,\ldots,\Nfe}$:
  \begin{subequations}
    \label{eq:weak-at-boundary}
    \begin{equation}
      \bigcup_{i=1}^{\ns}\A_{n-1,i}\cap\bigcup_{i=1}^{\ns}\I_{n,i} \subseteq \W_{n-1,\ns}\label{eq:weak-at-boundary:1},
    \end{equation}
    i.e. if for some $k$, $\exists i,j \in\cbrac{1,\ldots,\ns}$ such that $\lambda_{n-1,i,k}>0$ and $c_k{x_{n,j}} > 0$, then $c_k(x_{n-1,\ns}) = 0$ and $\lambda_{n-1,\ns,k} = 0$.
    \begin{equation}
      \bigcup_{i=1}^{\ns}\I_{n-1,i}\cap\bigcup_{i=1}^{\ns}\A_{n,i} \subseteq \W_{n-1,\ns}\label{eq:weak-at-boundary:2},
    \end{equation}
    i.e. if for some $k$, $\exists i,j \in\cbrac{1,\ldots,\ns}$ such that $\lambda_{n-1,i,k}>0$ and $c_k{x_{n,j}} > 0$, then $c_k(x_{n-1,\ns}) = 0$ and $\lambda_{n-1,\ns,k} = 0$. 
  \end{subequations}
\end{proposition}
\begin{proof}
  For compactness we write ${\A_n = \bigcup_{i=1}^{\ns}\A_{n,i}}$ and ${\I_n = \bigcup_{i=1}^{\ns}\I_{n,i}}$.
  
  \textit{\Cref{eq:weak-at-boundary:1}}: Assume that there exists a $k\in \A_{n-1}\cap\I_{n}$ such that $k\notin \W_{n-1,\ns}$.
  Recall that $k\notin \W_{n-1,\ns}$ implies $k\in A_{n-1,\ns}\cup \W_{n-1,\ns}$ or that either $\lambda_{n-1,\ns,k}>0$ or $c_k(x_{n-1,\ns})>0$.
  $\lambda_{n-1,\ns,k}>0$ yields a contradiction as there must exist an $i\in\cbrac{1,\ldots,\ns}$ such that $c_k(x_{n-1,i})>0$ which violates \cref{eq:cross-complementarity:a}.
  $c_k(x_{n-1,\ns})>0$ yields a contradiction as there must exist an $i\in\cbrac{1,\ldots,\ns}$ such that $\lambda_{n,i,k}>0$ which violates \cref{eq:cross-complementarity:b}.

  \textit{\Cref{eq:weak-at-boundary:2}}: Assume that there exists a $k\in \I_{n-1}\cap\A_{n}$ such that $k\notin \W_{n-1,\ns}$.
  $\lambda_{n-1,\ns,k}>0$ yields a contradiction as there must exist an $i\in\cbrac{1,\ldots,\ns}$ such that $c_k(x_{n,i})>0$ which violates \cref{eq:cross-complementarity:c}.
  $c_k(x_{n-1,\ns})>0$ yields a contradiction as there must exist an $i\in\cbrac{1,\ldots,\ns}$ such that $\lambda_{n-1,i,k}>0$ which violates \cref{eq:cross-complementarity:a}.
\end{proof}
%This proposition shows that the switches are isolated at the boundaries of the finite elements of variable length.
\subsection{Step-Equilibration}
\begin{table}[t]
  \small
  \centering
  \vspace{0.2cm}
  \begin{tabular}{r| c c c c c c c}
    &$\sigma_c^B$&$\sigma_c^F$&$\sigma_\lambda^B$&$\sigma_\lambda^F$&$\pi_c$&$\pi_\lambda$&$\nupds$\\\hline
    Entering $\partial\C$:        &$+$&$0$&$0$&$+$&$+$&$+$&$+$\\
    Leaving $\partial\C$:         &$0$&$+$&$+$&$0$&$+$&$+$&$+$\\
    No switch in $\partial\C$:    &$0$&$0$&$+$&$+$&$0$&$+$&$0$\\
    No switch in $\interior{\C}$: &$+$&$+$&$0$&$0$&$+$&$0$&$0$\\
  \end{tabular}
  \caption{Case decomposition of the construction of $\nupds$, for a fixed finite element.
    $+$ means that the value is strictly positive and $0$ that it is strictly zero.}
  \vspace{-0.5cm}
  \label{tab:nupds}
\end{table}
If there are no switches, then \cref{eq:cross-complementarity} is implied by the standard complementarity conditions in \cref{eq:discrete-dcs:comps}.
This means, there are spurious degrees of freedom in $h_n$ that must be removed.
To do this we introduce the switch-indicator variable $\nupds_n$ which is positive if there is a switch at the boundary between finite element $n$ and $n-1$ and zero otherwise.
We treat a variable $x>0$ as true and $x=0$ as false.
For $x,y,z \ge 0$, we define Boolean operators $\AND$ and $\OR$ using sets of constraints:
\begin{align*}
  &z = x \OR y \iff z\ge x,\ z\ge y,\ z \le x+y,\\
  &z = x \AND y \iff z \le x,\ z \le y,\ z\ge x+y-\max(x,y).
\end{align*}
The $\max$ operator is implemented via the Karush–Kuhn–Tucker conditions of the linear program:
%\begin{mini*}
%  {z\in \R}
%  {\hspace{-0.1cm}z}
%  {}
%  {\max(x,y) =}
%  \addConstraint{\hspace{-0.1cm}z}{\ge x,\ z\ge y.}
%\end{mini*}
\begin{align*}
	\max(x,y) = \min_{z\in \R} z \quad \mathrm{s.t.} \ {z\ge x,\ z\ge y.}
\end{align*}
For a fixed $n \in \cbrac{2,\ldots, \Nfe}$ we define:
\begin{subequations}
  \label{eq:sigmas}
  \begin{alignat}{2}
    \sigma_{c,n}^B &= \sum_{j=1}^{\ns}c\paren{x_{n-1, j}},\quad&&\sigma_{c,n}^F = \sum_{j=1}^{\ns}c\paren{x_{n, j}}\label{eq:sigmas:c},\\
    \sigma_{\lambda,n}^B &= \sum_{j=1}^{\ns}\lambda_{n-1, j},\quad&&\sigma_{\lambda,n}^F = \sum_{j=1}^{\ns}\lambda_{n, j},\label{eq:sigmas:lam}
  \end{alignat}
\end{subequations}
which are the backward and forward sums of $c(x)$ (\cref{eq:sigmas:c}) and $\lambda$ (\cref{eq:sigmas:lam}).
Intuitively these sums are analogous to the active set of the system in the finite element before and after the $n^{\mathrm{th}}$ boundary.
Using these we define, ${\pi_{\lambda,n} =\sigma_{\lambda,n}^B\OR\sigma_{\lambda,n}^F,}\ {\pi_{c,n}=\sigma_{c,n}^B\OR\sigma_{c,n}^F,}$ and ${\nupds_n= \pi_{c,n}\AND\pi_{\lambda,n}}$. 
The property we desire $\nupds_n$ to have is verified as shown in \Cref{tab:nupds} for the four possible conditions that occur at the boundary between two finite elements.
%, under the assumption that there are no weakly active constraints.
To enforce that the integration step size stays constant in the case where switches do not occur, for each $n = 2,\ldots \Nfe$, we introduce additional constraints with $M>>1$:
\begin{equation*}
  {-\nu_nM \le \paren{h_{n-1}- h_n} \le \nu_nM}
\end{equation*}
This formulation of $\nu_n$ can be seen as a mixed linear complementarity problem, rather than the nonlinear complementarity formulations in~\cite{Nurkanovic2022}.
\subsection{Direct Optimal Control with FESD}\label{sec:discrete-ocp}
We first define the following compact notation for the elements of a single control stage: $\mathbf{x}$ which collects all $x_{n,i}$, $\boldsymbol{\lambda}$ which collects all $\lambda_{n,i}$, and $\mathbf{h}=(h_1,\ldots,h_\Nfe)$.
We also collect the equations in \cref{eq:discrete-dcs}, \cref{eq:cross-complementarity}, and the step equilibration constraints into $G_{\mathrm{FESD}}(\mathbf{x},u,\boldsymbol{\lambda},\mathbf{h})$.
In order to solve \cref{eq:ocp} using FESD we discretize the time domain into $N$ uniformly sized control intervals with piecewise constant controls $\mathbf{u}=(u_1,\ldots,u_N)$ write the discrete-time optimal control problem in the form:
\begin{mini!}
  {\mathcal{X},\mathbf{u},\Lambda,\mathcal{H}}
  {\sum_{m=1}^{N}{L}_{m}(\mathbf{x}_m,u_m) + R(\mathbf{x}_{N,\ns})}
  {}
  {}
  \addConstraint{x_0}{=\bar{x}_0}
  \addConstraint{0}{= G_{\mathrm{FESD}}(\mathbf{x}_m,u_m,\boldsymbol{\lambda}_m,\mathbf{h}_m)}
  \addConstraint{0}{= \sum \mathbf{h}_m - \frac{T}{N}}
  \addConstraint{0}{\ge  g_{\mathrm{p}}(\mathbf{x}_m,u_m)}
  \addConstraint{0}{\ge  g_{\mathrm{t}}(\mathbf{x}_{N,\ns}),}
\end{mini!}
with variables $\mathcal{X}=(\mathbf{x}_1,\ldots,\mathbf{x}_N)$, $\Lambda= (\boldsymbol{\lambda}_1,\ldots,\boldsymbol{\lambda}_N)$, and $\mathcal{H}=(\mathbf{h}_1,\ldots,\mathbf{h}_N)$, the discretized Lagrange cost term ${L}_m: \R^{\Nfe\ns n_x}\times \R^{n_u}\rightarrow \R$, and Mayer term ${R: \R^{n_x}\rightarrow\R}$.
The Nonlinear Program (NLP) that results from this discretization is of a particular class called a Mathematical Program with Complementarity Constraints (MPCC).
This class of problems is degenerate as the complementarity constraints violate LICQ at all feasible points.
These MPCCs are formulated using CasADi~\cite{Andersson2019} in \texttt{nosnoc}~\cite{Nurkanovic2022b} and then solved via a series of related, relaxed, NLPs~\cite{Nurkanovic2023e}.
The NLPs are solved using IPOPT~\cite{Waechter2006}.
\section{Applications}
\begin{figure}[t]
  \centering
  \vspace{0.2cm}
  \begin{subfigure}{0.49\linewidth}
    \includegraphics[width=\linewidth]{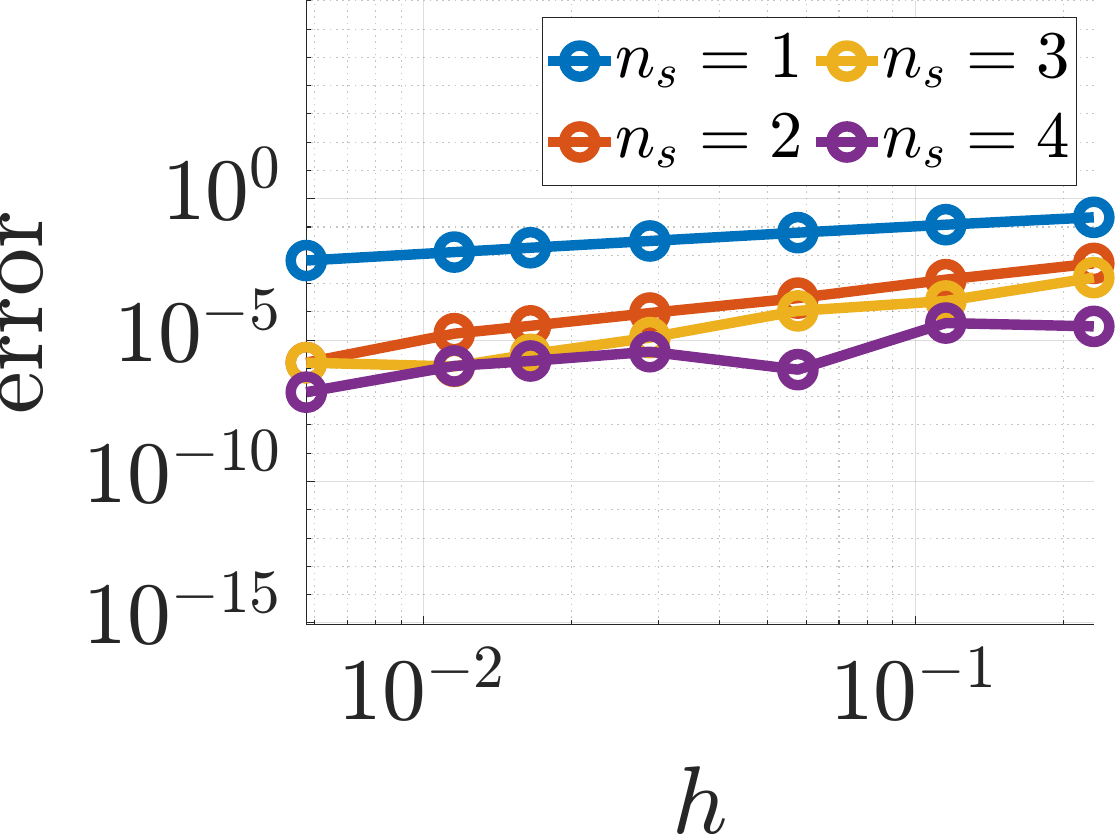}
    \caption{Fixed Time Stepping}
  \end{subfigure}
  \begin{subfigure}{0.49\linewidth}
    \includegraphics[width=\linewidth]{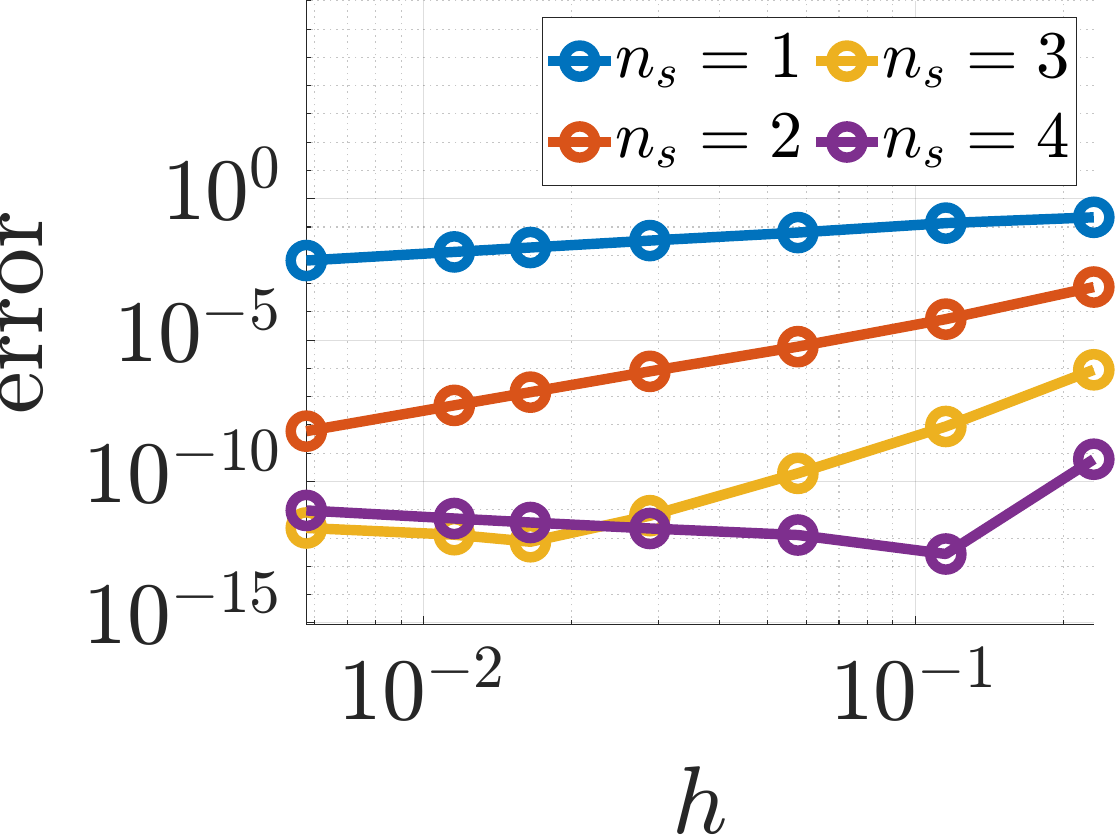}
    \caption{FESD}
  \end{subfigure}
  \vspace{-0.1cm}
  \caption{Plots of terminal error vs step size.}
  \vspace{-.5cm}
  \label{fig:order-plots}
\end{figure}

\begin{figure*}[t]
  \centering
  \vspace{0.2cm}
  \begin{subfigure}{0.19\linewidth}
    \includegraphics[width=\linewidth]{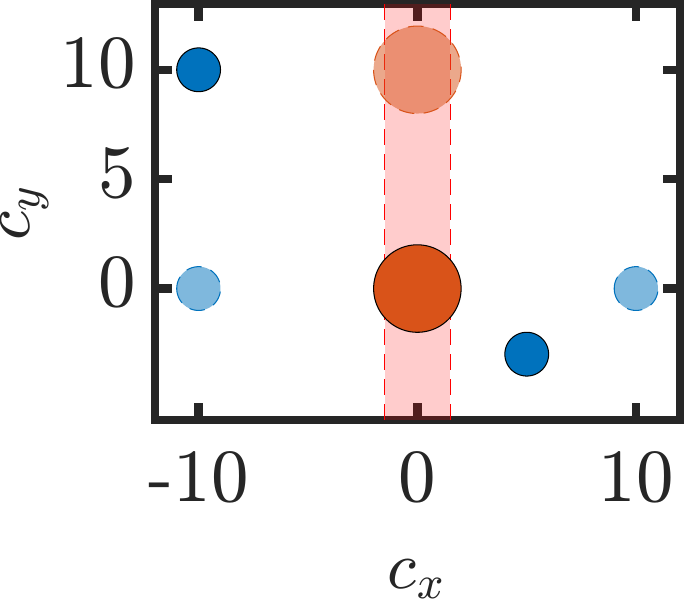}
    \caption{$t=0$}
    \label{fig:coop-discs:1}
  \end{subfigure}
  \begin{subfigure}{0.19\linewidth}
    \includegraphics[width=\linewidth]{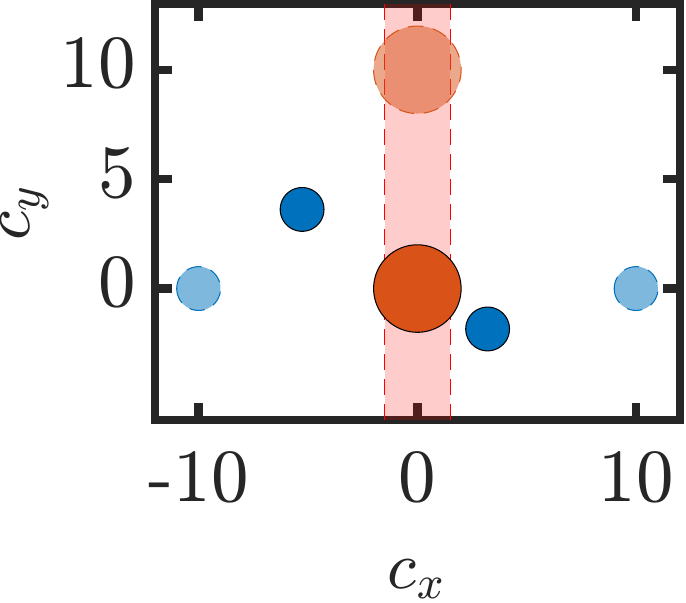}
    \caption{$t=0.8045$}
  \end{subfigure}
  \begin{subfigure}{0.19\linewidth}
    \includegraphics[width=\linewidth]{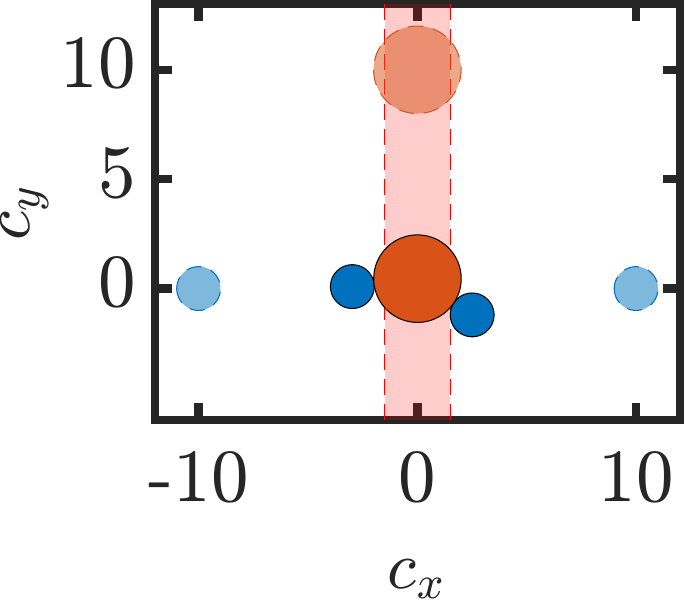}
    \caption{$t=1.3545$}
  \end{subfigure}
  \begin{subfigure}{0.19\linewidth}
    \includegraphics[width=\linewidth]{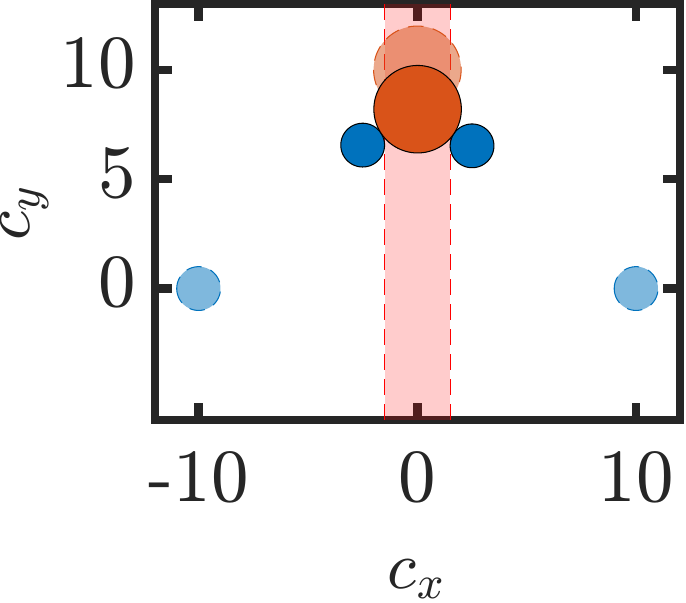}
    \caption{$t=3.3756$}
  \end{subfigure}
  \begin{subfigure}{0.19\linewidth}
    \includegraphics[width=\linewidth]{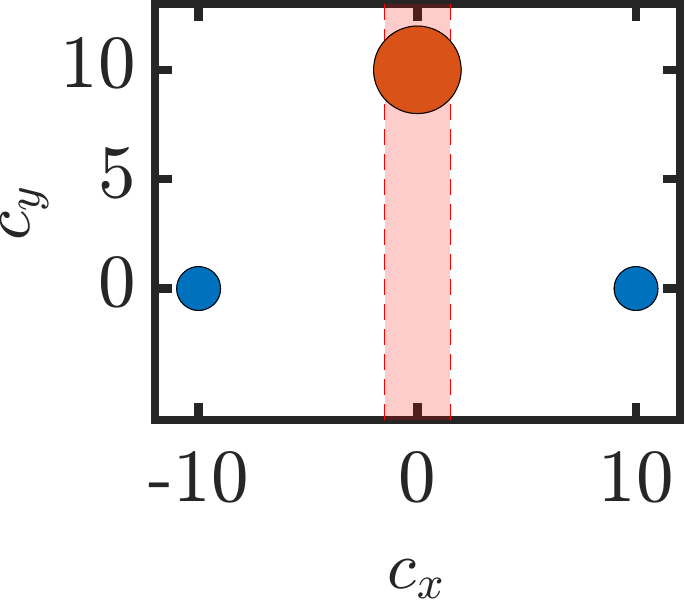}
    \caption{$t=5.0000$}
  \end{subfigure}
  \vspace{-0.2cm}
  \caption{Several frames of the solution for the collaborative manipulation problem.}
  \vspace{-0.8cm}
  \label{fig:coop-discs}
\end{figure*}
In this section, we first use the system in \Cref{ex:simple-pds} to verify the order accuracy of the FESD method applied to PDS.
Finally, we demonstrate an application of the method to a cooperative planar manipulation task formalized as an optimal control problem with the form in \cref{eq:ocp}.
These examples can be found on the branch \texttt{pds-examples} of \cite{Nurkanovic2022c}.

\subsection{Integration Order}
To verify the integration order accuracy of this extension to FESD, we use the simple PDS example described in \Cref{ex:simple-pds}.
We initialize the system at $t=0$ with $x_0 = \paren{\sqrt{2},\sqrt{2}}$ and simulate the system to a terminal time $T =\frac{11\pi}{12} + \sqrt{3}$.
We can show that the analytical solution for the terminal state $x(T) = \paren{-1,0}$ by decomposing the trajectory into three smooth sections: the initial circle of radius two, the sliding mode from the intersection of that circle with $x_2=-1$, and the circle of radius one which begins after we leave the sliding mode at $x = (0,-1)$.
The first of these sections occurs on the intervale $t\in \left[0,\frac{5\pi}{12}\right)$, the second on the interval $t\in \left[\frac{5\pi}{12}, \frac{5\pi}{12}+\sqrt{3}\right)$, and the third on $t\in \left[\frac{5\pi}{12}+\sqrt{3}, T\right]$.
The plot of this trajectory is shown in \Cref{fig:ex1-sol}.

This system was simulated with the software package \texttt{nosnoc} on the interval $[0,T]$ with the total $\Nfe = 10, 20, 40, 80, 140, 200, 400$, and an underlying Radau-IIA implicit RK scheme with $\ns=1,2,3,4$ stage points.
We calculate the terminal error as $\norm{x(T)-(-1, 0)}$.
The results are shown in \Cref{fig:order-plots}, in which the standard RK discretization without FESD only achieves between first and second-order accuracy as the step size is reduced, while the FESD discretization achieves the expected accuracy $O(h^{2\ns-1})$.
These results match various prior work on time-stepping methods (without event detection) applied to a variety of complementarity formulations of constrained dynamical systems where accuracy is limited to $O(h^2)$ independent of the order of the underlying schema used~\cite{Loetstedt1984}.
\subsection{Collaborative Manipulation}
We present a collaborative manipulation problem in which the controller must use two actuators with limited range to transport an unactuated disc.
The PDS formulation of this problem is inspired by the non-overlapping disc constraints of the first-order sweeping process models described in~\cite{Cao2021,Mordukhovich2023}.
The state and controls of the system are $x=(x_1,x_2,x_3)\in\R^6$, and $u=(u_1,u_2)\in\R^4$.
The smooth ODE is $f(x,u) = (u_1,u_2,0,0)$ and set $\C$ defined by $c_j(x) =\norm{x_3(t) - x_j(t)}^2 - (R_3 + R_j)^2$ with ${R_1=1}$, ${R_2=1}$, ${R_3=2}$, for ${j =1,2}$.
The system starts at ${x_0= \transp{(-10,10,5,-5,0,0)}}$.
With this, we formulate the OCP:
\begin{mini*}
  {\substack{x(\cdot). u(\cdot)}}
  {\int_{0}^{T}\!\transp{u(t)}Ru(t) \dd t\! + \! \transp{(x(T)-x_T)}Q_T(x(T)-x_T)}
  {}
  {}
  \addConstraint{x(t)}{=\proj{\tancone{\C}{x(t)}}{f(x(t),u(t))},\ }{t\in[0,T]}
  \addConstraint{x(0)}{=x_0}
  \addConstraint{-10\le u_i}{\le 10}
  \addConstraint{x_{1,1}}{\le -2.5}
  \addConstraint{x_{2,1}}{\ge 2.5.}
\end{mini*}
We use the weighting matrices $R = \mathrm{diag}(10^{-4}, 10^{-4})$, $Q_T = \mathrm{diag}(1,1,1,1,10^3,10^3)$, and target positions ${x_T=(-10,0,10,0,0,10)}$.
The additional constraints on the state enforce that neither actuator can individually accomplish the goal.
The initial configuration is shown in \Cref{fig:coop-discs:1}, where the solid blue and orange discs are the initial locations, the transparent discs represent the target configuration and the red region is the zone that any part of the actuators, the blue discs, may not enter.
This OCP is then discretized with FESD with a Radau-IIA scheme with $\ns = 2$.
We display several frames of the optimal solution in \Cref{fig:coop-discs}.
A video of several other solutions to this kind of manipulation problem can be found at \href{https://youtu.be/HXHAbjxC6rw}{https://youtu.be/HXHAbjxC6rw}.
\section{Conclusions and Future Work}
This paper provides an extension of the Finite Elements with Switch Detection (FESD) method to Projected Dynamical Systems (PDS) and a certain class of first-order sweeping processes.
We show that this method recovers the accuracy of higher-order Runge-Kutta discretizations and shows its application to optimal control with a planar manipulation task.
In future extensions, we would like to further tackle first-order Sweeping processes with moving, time-dependent sets.
\bibliographystyle{ieeetran}
\bibliography{syscop}
\end{document}